\documentclass[14pt]{article}
\usepackage[latin9]{inputenc}
\usepackage[T1]{fontenc}
\usepackage[french,english]{babel}
\usepackage[hidelinks]{hyperref}
\usepackage{graphicx}
\usepackage[top=3cm, bottom=3cm, left=3cm, right=3cm]{geometry}
\usepackage{float}
\usepackage{amssymb}
\usepackage{amsmath}
\usepackage{titlepic}
\usepackage{caption} 
\usepackage{epstopdf}
\usepackage{xfrac}
\usepackage{enumitem}
\usepackage{subfigure}
\usepackage{bm}
\usepackage{esint}
\usepackage{mathtools}
\usepackage[dvipsnames]{xcolor}
\usepackage{dsfont}
\usepackage{mathrsfs}
\usepackage{tikz}

\usepackage{amsthm}

\usepackage{todonotes}

\DeclareRobustCommand{\rchi}{{\mathpalette\irchi\relax}}
\newcommand{\irchi}[2]{\raisebox{\depth}{$#1\chi$}}

\newtheorem{theorem}{Theorem}[section]

\newtheorem{corollary}{Corollary}[section]

\newtheorem{lemma}{Lemma}[section]

\newtheorem{proposition}{Proposition}[section]

\usepackage{relsize}

\begin{document}
	
	\title{Early-warning inverse source problem for the elasto-gravitational equations}
	\author{Lorenzo Baldassari\thanks{\footnotesize CMOR Department, Rice University, Houston, TX 77005, USA.} \and Maarten V. de Hoop\thanks{\footnotesize Computational and Applied Mathematics and Earth Science, Rice University, Houston, TX 77005, USA.} \and Elisa Francini\thanks{\footnotesize Dipartimento di Matematica e Informatica `U. Dini', Universit\`a di Firenze, Firenze FI, Italy.} \and Sergio Vessella\footnotemark[3]}
	\maketitle
	
	\begin{abstract}
		Through coupled physics, we study an early-warning inverse source problem for the elasto-gravitational equations. It consists of a mixed hyperbolic-elliptic system of partial differential equations describing elastic wave displacement and gravity perturbations produced by a source in a homogeneous bounded medium. Within the Cowling approximation, we prove uniqueness and Lipschitz stability for the inverse problem of recovering the moment tensor and the location of the source from early-time measurements of the changes of the gravitational field. The setup studied in this paper is motivated by gravity-based earthquake early warning systems, which are gaining much attention recently.
	\end{abstract}

\def\keywords{\vspace{.5em}{\textbf{\;\;  Keywords:}~\,\relax}}
\def\endkeywords{\par}
\keywords{inverse problems, Lipschitz stability, elastodynamic systems}

\def\keywords2{\vspace{.5em}{\textbf{\;\;  Mathematics Subject Classification (MSC2000):}~\,\relax}}
\def\endkeywords2{\par}
\keywords2{35R30, 35Q86, 35J05, 35L10.}

\section{Introduction}\label{sec:introduction}
\subsection{Definition of the problem}
In this paper we study, through coupled physics, an early-warning inverse source problem for the elasto-gravitational equations  motivated by seismology. It consists of a mixed hyperbolic-elliptic system of partial differential equations describing elastic wave displacement and gravity perturbations produced by a source in a homogeneous bounded medium. 

Consider the following Cauchy problem with Neumann boundary condition for the elastic equation:
\begin{equation} 
	\begin{cases}
		\rho_0 \textbf{u}_{tt} - \text{div} (\mathbb{C}\nabla \textbf{u})  = \textbf{f} & x \in \Omega \times [0,\infty),\\
		(\mathbb{C}\nabla \textbf{u}) \cdot \nu = 0 & (x,t) \in \partial \Omega\times [0,\infty), \\
		\textbf{u}(x,0) = \textbf{u}_t(x,0) = 0 & x \in \Omega,\\
	\end{cases} 
	\label{(1)}
\end{equation}
where $\Omega$ is a bounded domain in  $\mathbb{R}^3$, $\textbf{u} = \textbf{u}(x,t)$ denotes the displacement, $\nu$ is the outward normal to $\partial \Omega$, $\rho_0>0$ denotes the constant density of the medium, and $\mathbb{C}$ is the isotropic stiffness tensor with constant Lam\'e parameters $\lambda_0,\mu_0 >0$:
\[
C_{ijk\ell} = \lambda_0 \delta_{ij} \delta_{k\ell} + \mu_0 (\delta_{ik} \delta_{j\ell} + \delta_{i\ell} \delta_{jk}).
\]	
The source term $\textbf{f}$  is defined by 
\begin{equation*}
\textbf{f}=-M \nabla (q(|x-P|)),
\end{equation*}
where $M$ is a constant $3\times 3$ real valued matrix, $P \in \Omega$ and $q(|\cdot|) \in C^2_0(\Omega)$.

Here we study uniqueness and stability for the early-warning inverse source problem that consists in recovering the moment tensor $M$ and the location $P$ of the source from early-time measurements of the changes of the gravitational field $\nabla S^+$ generated according to the following transmission problem for the Newtonian Poisson's equation:
\begin{equation} 
	\begin{cases}
		\Delta S^- = -\rho_0\text{div}(  \textbf{u}) & (x,t) \in \Omega \times [0,\infty),\\
		\Delta S^+ = 0 & (x,t) \in (\mathbb{R}^3 \setminus \overline{\Omega})\times [0,\infty), \\
		S^{-} = S^{+} & (x,t) \in \partial \Omega\times [0,\infty), \\
		(\nabla S^{-} + \textbf{u}) \cdot \nu  = \nabla S^{+} \cdot \nu & (x,t) \in \partial \Omega\times [0,\infty), \\
		S^+\to 0 & |x|\to \infty,
	\end{cases} 
	\label{(2)}
\end{equation}
where $\textbf{u}$ solves \eqref{(1)}. The early-time measurements of $\nabla S^+$ are taken on an open ball. The coupling between the elastic equation and Poisson's equation is known in the literature as the {elasto-gravitational coupling}.

The setup above is motivated by gravity-based EEW
(earthquake early warning) systems: $\Omega$ is meant to model the Earth, while $\textbf{f}$ represents a double-couple source approximating an earthquake source with seismic moment tensor $M$ and location $P$. In particular, the formulation that we follow here is based on the so-called non-self-gravitating model, which geophysicists have first used when they started studying gravity-based EEW a few years ago.

\subsection{Context}

EEW systems rapidly detect and estimate the magnitude of ongoing earthquakes in real time to provide advance warnings of impending ground motion \cite{allen2019}.  Conventional EEW systems rely on detecting the P-elastic waves, whose finite speed of propagation imposes a minimum on the warning time.
This is key, since EEW systems may fail to rapidly estimate the size of large offshore subduction earthquakes, like the 2011 Tohoku earthquake, due to the slowness of the elastic waves \cite{meier2017, ampuero2017, minson2018, wald2020}. 
Recently discovered PEGS (prompt elasto-gravity signals) have raised hopes that these limitations may be overcome \cite{montagner2016prompt, vallee2017observations}. PEGS are earthquake-associated signals created by density-perturbation-induced gravity field, and by the associated elastic readjustment of the gravitationally perturbed Earth \cite{ringler2022}. PEGS are readily present in the self-gravitating equations governing the earthquake-induced motion \cite{dahlen2020theoretical}, but their observations have only been provided recently, in the occurrence of large earthquakes, by ground-based seismometers \cite{vallee2019multiple, zhang2020prompt}. Since these signals are transmitted at the speed of light everywhere on Earth, the earliest deformation signals are not expected to be carried by the fastest P-elastic wave, but by PEGS \cite{harms2015transient,harms2016transient}. This is why gravity-based EEW systems have been gaining a lot of attention in recent years \cite{licciardi2022instantaneous, mendezrapid}. Including PEGS in early warning systems is expected to result in faster detection of large earthquakes, especially when compared to conventional P-elastic wave-based EEW systems \cite{licciardi2022instantaneous, mendezrapid}.

While the physics of PEGS is understood, their use lacks a mathematical justification. In other words, we still need to develop the theoretical framework of the PEGS inverse problem, with uniqueness and stability theorems and proofs. The main challenge here is that PEGS are present in the self-gravitating equations \cite[Section 3.3.2]{dahlen2020theoretical}, which differ from the ones studied in this paper by the presence of a non-local term, 
\[
\nabla S(x,t) = - \rho_0 \int_{ \Omega} \frac{\text{div}( \textbf{u}(y,t))}{|x-y|}\, \text{d}y,
\] 
added to the elastic equation: 
\begin{equation*}
	\rho_0 \textbf{u}_{tt}(x,t) - \text{div}(\mathbb{C} \nabla \textbf{u}(x,t)) + \rho_0 \nabla S(x,t) = \textbf{f}(x,t), \qquad (x,t) \in \Omega\times [0,\infty).
	\label{(3)}
\end{equation*}
However, as explained in \cite{harms2015transient, vallee2017observations, juhel2019normal}, when measurements are not taken with a ground-based seismometer (one can use, instead, future generation gravity strainmeters \cite{juhel2018,  shimoda2021early, zhang4245177seismic}), it is not completely unrealistic to neglect the effects of the early-time ground-induced motion due to the non-locality of $\rho_0 \nabla S$. 
Since our intention in this paper is to focus on the basic aspects of the inverse problem arising in gravity-based EEW systems, a way to go is to account for the effects of gravity within the Cowling approximation \cite{cowling1941non}, that is by ignoring self-gravitation effects. 
In fact, according to \cite[section 4.3.5]{dahlen2020theoretical}, the self-gravitation is expected to be significant only at periods longer than 100s, whereas our primary interest here is the short timescales over which gravity perturbations may develop \cite{harms2015transient}. 

The non-self-gravitating model proposed in this paper is thus still valuable for gaining insight into the role of the gravitational perturbations generated by Poisson's equation in the inverse problem. It also facilitates the mathematical analysis of the elastic and gravitational components of the system, making it possible to study the elastic equation  independently from Poisson's equation.

\subsection{Contribution and organization of the paper}

In this paper we prove, under suitable a priori assumptions, uniqueness and Lipschitz stability of the early-warning inverse source problem. We make use of the following tools:
\begin{itemize}
	\item Energy estimates for the elastic equation.
	\item Estimate of propagation of smallness for elliptic equations, see \cite{Alessandrini_2009}.
\end{itemize}
The main idea is to turn the elasto-gravitational coupling to our advantage, since the changes of the gravitational field $\nabla S^+$ are generated {instantaneously} by Poisson's equation, and thus can be used to solve the early-warning inverse source problem without having to wait for the elastic waves to reach the boundary of $\Omega$.

The paper is structured as follows. In paragraph 1.4, we introduce notation that will be used throughout the work. Section 2 contains the description of the direct problem. In paragraph 2.1, we prove existence and uniqueness of the weak solutions to the elasto-gravitational equations. In paragraph 2.2 we give some energy estimates concerning the solutions to problem \eqref{(1)}. In section 3, we formulate our inverse problem rigorously. We then prove the main results of this paper: uniqueness first, and then Lipschitz stability.

	\subsection{Notation} 
We shall denote by $B_{r}(x)$ the open ball in $\mathbb{R}^3$ of radius $r$ and center $x$. We shall use the abbreviation  $B_r$ when the center is the origin. 

Here and in the next sections we shall assume that the Earth is represented by a nonempty open bounded convex set $\Omega$ in $\mathbb{R}^3$ with $C^2$ boundary, i.e., locally it can be written as the graph of a $C^2$ function on $\mathbb{R}^2$. 

For any  $h>0$, let us define the set 
\[
\Omega_h := \{x \in \Omega \; | \; \text{dist}(x,\partial \Omega)>h \}.
\]

Given a function $\textbf{u}:\Omega \times [0,\infty) \to \mathbb{R}^3$, $\textbf{u} = \textbf{u}(x,t)$, we shall denote by $\partial_j u_i$ and $u_{i,t}$ the derivatives of the $i$-th component of $\textbf{u}$ with respect to the $x_j$ variable and to the time $t$, respectively, and similarly for higher order derivatives. 

We shall also denote by $\mathbb{M}^m$ the space of $m\times m$ real valued matrices and by $\mathcal{L}(X,Y)$ the space of bounded linear operators between Banach spaces $X$ and $Y$. 

For every matrices $A, B \in \mathbb{M}^m$ and for every $\mathbb{L} \in \mathcal{L}(\mathbb{M}^m,\mathbb{M}^m)$, we shall use the following notation:
\begin{equation*}
	\begin{aligned} 
		& (\mathbb{L} A)_{ij} = L_{ijk\ell} A_{k\ell},\\
		& A \cdot B = A_{ij} B_{ij}, \\
		& |A| = \sqrt{(A \cdot A)}.\\
	\end{aligned}
\end{equation*} 
Notice that here and in the sequel summation over repeated indexes is implied.

	\section{The direct problem}\label{sec:forward}
	We are now ready to present the direct or forward problem and prove its solvability.

	We first make some assumptions. We assume that the Earth is made of inhomogeneous linear elastic material. We denote by $\mathbb{C}(x) \in \mathcal{L}(\mathbb{M}^3,\mathbb{M}^3)$ the isotropic stiffness tensor with Lam\'e parameters $\lambda,\mu \in C^1(\overline{\Omega})$
\[
C_{ijk\ell}(x) = \lambda(x) \delta_{ij} \delta_{k\ell} + \mu(x) (\delta_{ik} \delta_{j\ell} + \delta_{i\ell} \delta_{jk}),
\]	
and by $\rho \in C^1(\overline{\Omega})$
the reference density. We assume that
\begin{equation*}
	\begin{aligned} 	
		& \lambda(x) \geq \lambda_0 > 0, \quad \mu(x) \geq \mu_0 > 0, \quad \rho(x) \geq \rho_0 > 0 \quad \forall x\in \overline{\Omega}.
	\end{aligned} 
\end{equation*}
Before going further, we need to underline that, while the well-posedness of the direct problem can be proven with variable coefficients, the uniqueness and stability theorems for the inverse problem in Section \ref{sec:inverse} will require $\lambda, \mu$ and $\rho$ to be constant. In such case, we will say that \[\lambda = \lambda_0, \quad  \mu=\mu_0, \quad \rho=\rho_0 \quad \text{in } \overline{\Omega}.\]

We model the source by a function $\textbf{f}$  with regularity at least $H^1(\Omega,\mathbb{R}^3)$, which approximates  the body force defined in \cite{dahlen2020theoretical, de2015system} with $t_0=0$. As mentioned in section \ref{sec:introduction}, we define
\[
\textbf{f}=-M \nabla (q(|x-P|)),
\]
where 
\begin{equation}
	P \in \Omega_{d_0}
	\label{def:P}
\end{equation} 
is the location of the source and $M \in \mathbb{M}^3$ is the moment tensor. We assume that $M$ is nonzero, symmetric, with vanishing trace,
see \cite[section 2.8]{de2015system} for details. Here and in the next sections, we also assume that $q(|\cdot|) \in C_0^2(\Omega)$, and \begin{equation}
	\text{supp }q \subset \left[-\frac{d_0}{2},\frac{d_0}{2}\right], \qquad \int_{\mathbb{R}^3} q(|x|)= 1. 
	\label{def:q}
\end{equation}	

The direct problem for the non-self-gravitating equations consists in finding the solution pair $(\textbf{u}, S)$, where $\textbf{u}$ solves the following Cauchy problem with Neumann boundary condition for the elastic equation
\begin{equation} 
	\begin{cases}
		\rho \textbf{u}_{tt} - \text{div} (\mathbb{C}\nabla \textbf{u})  = \textbf{f} & x \in \Omega \times [0,\infty),\\
		(\mathbb{C}\nabla \textbf{u}) \cdot \nu = 0 & (x,t) \in \partial \Omega\times [0,\infty), \\
		\textbf{u}(x,0) = \textbf{u}_t(x,0) = 0 & x \in \Omega,\\
	\end{cases} 
	\label{eq:elastic-dp}
\end{equation}
and $S$ is the solution to the following transmission problem for Poisson's equation
			\begin{equation} 
	\begin{cases}
		\Delta S^- = -\text{div}(\rho  \textbf{u}) & (x,t) \in \Omega \times [0,\infty),\\
		\Delta S^+ = 0 & (x,t) \in (\mathbb{R}^3 \setminus \overline{\Omega})\times [0,\infty), \\
		S^{-} = S^{+} & (x,t) \in \partial \Omega\times [0,\infty), \\
		(\nabla S^{-} + \textbf{u}) \cdot \nu  = \nabla S^{+} \cdot \nu & (x,t) \in \partial \Omega\times [0,\infty), \\
		S^+\to 0 & |x|\to \infty.
	\end{cases} 
	\label{eq:poisson-dp}
\end{equation}
It is evident that, for the non-self-gravitating equations, we can solve first \eqref{eq:elastic-dp}, and then \eqref{eq:poisson-dp}. 

\subsection{Existence, uniqueness, and regularity of solutions}
As anticipated, we will first show existence, uniqueness, and regularity of solutions to problem \eqref{eq:elastic-dp}. The proofs are based on results of \cite{lagnese}.

\begin{theorem}\label{t:existence-uniqueness} \label{t:regularity}
Assume that $\textup{\textbf{f}} \in L^2(\Omega, \mathbb{R}^3)$. Then there is a unique weak solution \textup{\textbf{u}} to \eqref{eq:elastic-dp} such that
\begin{equation*}
	\textup{\textbf{u}} \in C([0,\infty); H^1(\Omega, \mathbb{R}^3)),\quad
	\textup{\textbf{u}}_t \in C([0,\infty); L^2(\Omega, \mathbb{R}^3)).
\end{equation*}
If $\textup{\textbf{f}} \in H^1(\Omega)$, then the solution to \eqref{eq:elastic-dp}
		is such that
		\begin{equation*}
			\textup{\textbf{u}} \in C([0,\infty); H^2(\Omega, \mathbb{R}^3)),
		\end{equation*}
		and satisfies \eqref{eq:elastic-dp} in a pointwise sense. 
		Also
		\begin{equation*}
			\textup{\textbf{u}}_t \in C([0,\infty); H^1(\Omega, \mathbb{R}^3) ), \quad \textup{\textbf{u}}_{tt} \in C([0,\infty); L^2(\Omega, \mathbb{R}^3) ).
		\end{equation*}
	\end{theorem}
\begin{proof} 
By the Duhamel's principle, \textbf{u} can be represented as follows:
	\begin{equation*}
		\textbf{u}(x,t) = \int_0^t \textbf{w}(x,t;s) \, \text{d}s, 
		\label{eq:formula-duhamel}
	\end{equation*}
	where $\textbf{w}$ solves	\begin{equation*}
		\begin{cases}
			\rho \textbf{w}_{tt} - \text{div} (\mathbb{C} \nabla \textbf{w}) = 0  & (x,t) \in \Omega \times (s,\infty),  \\
			(\mathbb{C}\nabla \textbf{w}) \cdot \nu = 0 & (x,t ) \in \partial \Omega \times [s,\infty) \\
			\textbf{w}(x,s;s ) = 0 & x \in \Omega, \\
			\textbf{w}_t(x,s;s) = \textbf{f}(x) & x \in \Omega.
		\end{cases}
		\label{eq:duhamel}
	\end{equation*}
We can now easily deduce existence, uniqueness, and regularity results for  \eqref{eq:elastic-dp} from \cite[Theorems 2.1 and 2.2]{lagnese} applied to $\textbf{w}$.

\end{proof}

	We now consider the standard transmission problem for Poisson's equation of the gravitational potential
	\begin{equation} 
		\begin{cases}
			\Delta S^- = -\text{div}(\rho  \textbf{u}) & (x,t) \in \Omega \times [0,\infty),\\
			\Delta S^+ = 0 & (x,t) \in (\mathbb{R}^3 \setminus \overline{\Omega})\times [0,\infty), \\
			S^{-} = S^{+} & (x,t) \in \partial \Omega\times [0,\infty), \\
			(\nabla S^{-} + \textbf{u}) \cdot \nu  = \nabla S^{+} \cdot \nu & (x,t) \in \partial \Omega\times [0,\infty),\\
			S^+ \to 0 & |x| \to \infty.
		\end{cases}
		\label{eq:poisson}
	\end{equation}	Let us set
	\[
	S(x,t) = 
	\begin{cases} 
		S^-(x,t) & (x,t) \in \Omega \times [0,\infty),\\
		S^+(x,t) & (x,t) \in \mathbb{R}^3 \setminus \overline{\Omega} \times [0,\infty). 
	\end{cases}  
	\] 	
	Problem \eqref{eq:poisson} can be written in weak formulation as
	\begin{equation}
		\text{find }S\in H^1(\mathbb{R}^3)\text{ such that:}\quad \forall \phi \in H^1(\mathbb{R}^3), \int_{\mathbb{R}^3} \nabla S \cdot \nabla \phi  = -\int_{\Omega} \rho \textbf{u} \cdot \nabla \phi.
		\label{eq:weak-poisson}
	\end{equation}
	The  Lax-Milgram theorem \cite{evans1998partial} implies the existence of a unique $S \in H^1(\mathbb{R}^3)$ solving \eqref{eq:weak-poisson}.
	
	\subsection{Energy estimates}
	
	We will use the following energy estimates, whose proofs are contained in Appendix  \ref{sec:appA}. For the sake of simplicity we formulate such estimates for $\lambda = \lambda_0, \mu=\mu_0$ and $\rho=\rho_0$ in $\overline{\Omega}$, but they hold true also for $\lambda,\mu$, and $\rho \in C^1(\overline{\Omega})$.
	
	\begin{proposition}\label{p:energy-1}The solution $\textup{\textbf{u}}$ to \eqref{eq:elastic-dp} satisfies the inequality
		\begin{equation} \int_{\Omega} |\textbf{\textup{u}}(\cdot, \tau)|^2\, \leq \tau^3 \frac{e^{\tau/\rho_0}}{\rho_0} \int_{\Omega} |\textbf{\textup{f}}|^2,
		\end{equation} 
		for any $\tau\geq 0$.
	\end{proposition} 
	
	\begin{proposition}\label{p:energy}
		Let $\alpha>0$, $x_0 \in \partial \Omega$ and $\tau >0$. Denote by $K(x_0,\tau,\alpha)$ the cone
		\[K(x_0,\tau,\alpha) := \{(x,t) \in \mathbb{R}^4 \text{ such that } 0\leq t\leq \tau - \alpha |x-x_0| \},
		\]
		and define 
		\[
		\widetilde{K}(x_0,\tau,\alpha) := K(x_0,\tau,\alpha) \cap  (\Omega \times [0, \infty)).  
		\]
		For \[\alpha=\alpha_0:=\sqrt{\frac{\rho_0}{2 (\lambda_0 +2\mu_0)}},\] the solution $\textup{\textbf{u}}$ to \eqref{eq:elastic-dp} satisfies the inequality
		\begin{equation}
			\int_{\widetilde{K}} |\textup{\textbf{u}}|^2 e^{-2\gamma t}\leq \left(\frac{\tau}{\rho_0 \gamma}\right)^2 \int_{\widetilde{K}} |\textup{\textbf{f}}|^2 e^{-2\gamma t},
			\label{eq:energy-est}
		\end{equation}
		for	any $\gamma>0$.
	\end{proposition}
	
The following corollaries of Proposition \ref{p:energy} will be useful later. An illustration of them is given in Figure \ref{fig:corollaries}.
	\begin{corollary}\label{c:null-cone}
		Define
		\[
		\widetilde{\mathcal{K}}(\tau) := \bigcup\limits_{x_0 \in \partial \Omega} \widetilde{K}(x_0,\tau,\alpha_0).
		\]
		There exists $\tau_0 >0$ such that, for any $\tau \in(0,\tau_0)$, $\textup{\textbf{u}}=0$ in $\widetilde{\mathcal{K}}(\tau)$.
	\end{corollary}

	\begin{proof}
		Since supp $\textbf{f} \subset \Omega_{\frac{d_0}{2}}$ by \eqref{def:P}, it suffices to take 
		\[
		\tau_0=\frac{\alpha_0 d_0}{2}
		\]
		to have that, for any $\tau\in(0,\tau_0)$, 
		\[\text{supp }\textbf{f} \cap \widetilde{K}(x_0, \tau,\alpha_0)=\emptyset,
		\] 
		independently from the choice of $x_0 \in \partial \Omega$. 
		Using \eqref{eq:energy-est}, we obtain that $\textbf{u}=0$ in $ \widetilde{\mathcal{K}}(\tau)$.

	\end{proof}
	

	\begin{corollary}\label{c:neigh}
			Let $\tau_0$ be the same as in Corollary \ref{c:null-cone}. 		We have \begin{equation*}\textup{\textbf{u}}=0 \quad \text{on } \partial \Omega \times [0,\tau_0).
			\label{eq:boundary-null}
		\end{equation*}
		 Also, for any $\tau \in \left(0, \tau_0\right)$, we have
		 \[
		 \Omega \setminus \overline{\Omega}_{\frac{d_0}{2} - \frac{\tau}{\alpha_0}}\times [0,\tau] \subseteq \widetilde{\mathcal{K}}(\tau),
		 \]
		 hence $\textup{\textbf{u}} =0$ there.
	\end{corollary}


%

	\begin{figure}
		\centering
		\includegraphics[scale=1.1]{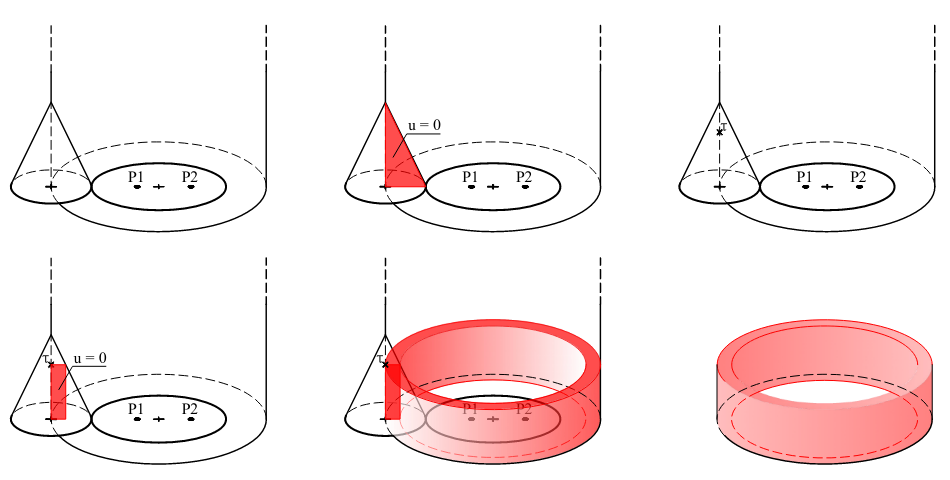}
		\caption{An illustration of Corollaries \ref{c:null-cone} and \ref{c:neigh}. Notice that $\textbf{u}=0$ in the red areas. This illustration shows how to construct $\Omega \setminus \overline{\Omega}_{\frac{d_0}{2} - \frac{\tau}{\alpha_0}}\times [0,\tau]$ (cf. the last picture) such that it is contained in $\widetilde{\mathcal{K}}(\tau)$.}
		\label{fig:corollaries}
	\end{figure}
	
	\section{The inverse problem}\label{sec:inverse}
		
	In this section, we will investigate the following inverse problem: 
	\textit{given early-time measurements of the changes of the gravitational field $\nabla S^+$ generated by the source, can we determine uniquely and in a stable way the moment tensor $M$ and the location $P$ of the source?} 
		
	We suppose	
		to have measured $\nabla S^+$ on $B_{r_0}(\bar{x})$ contained in $\mathbb{R}^3 \setminus \overline{\Omega}$, for times in the range $t \in [0,t_0]$. To prove uniqueness and stability for the early-warning inverse problem we need to assume that the Lam\'e parameters and the density are positive constants
		\[
		\lambda = \lambda_0, \quad \mu = \mu_0, \quad \rho = \rho_0  \quad \text{in } \overline{\Omega}.
		\]
In addition, to prove the Lipschitz stability estimate, we need to assume that all the admissible moment tensors $M$ satisfy the condition
\begin{equation}
m_0 \leq |M| \leq M_0
,
\label{eq:M-bound}
\end{equation}
given $m_0, M_0 >0$.

Before going further, it is noteworthy to point out that the early-warning inverse problem we propose in this paper can be solved for any $t_0>0$, hence the adjective ``early-warning''. This is an advantage with respect to conventional inverse seismic problems, which are usually based on elastic waves. In fact, since the elastic waves propagate at finite speed, one has to wait for them to reach the boundary of $\Omega$ (see, for example, Theorem 4.2 in \cite{de2023quantitative}). The changes of the gravitational field $\nabla S^+$, instead, are generated {instantaneously} by Poisson's equation, and thus can be used to solve the inverse problem without having to impose a minimum on the time needed to determine uniquely the source. In this sense, our uniqueness and stability results can be viewed as a first step towards a mathematical justification of the benefits of using gravity-based EEW systems. 
	
	\subsection{Statement of the main results}
	
	The main result of this paper is the Lipschitz stability of the inverse problem:
	
	\begin{theorem}[Lipschitz stability]\label{t:stability}
Let $B_{r_0}(\bar{x}) \subset \mathbb{R}^3 \setminus \overline{\Omega}$, $t_0>0$. Consider two sources, $\textup{\textbf{f}}^{(1)}$ and $\textup{\textbf{f}}^{(2)}$, such that \[\textup{\textbf{f}}^{(j)} = - M^{(j)} \nabla( q (|x-P^{(j)}|)), \quad j=1,2,\]
		where $M^{(1)}, M^{(2)} \in \mathbb{M}^3$ are nonzero, symmetric, with vanishing trace, satisfy \eqref{eq:M-bound},  $P^{(1)},P^{(2)} \in \Omega_{d_0}$, and $q \in C_0^2(\Omega)$ satisfies \eqref{def:q}.  
		
		Let $(\textup{\textbf{u}}^{(1)}, S^{(1)})$ and $(\textup{\textbf{u}}^{(2)}, S^{(2)})$
		be (weak) solutions to \eqref{eq:elastic-dp}-\eqref{eq:poisson-dp} associated to $\textup{\textbf{f}}^{(1)}$ and $\textup{\textbf{f}}^{(2)}$, respectively, when $\lambda=\lambda_0$, $\mu=\mu_0$ and $\rho=\rho_0$. If   
		\[\|\nabla S^{(1)} (\cdot, t)- \nabla S^{(2)}(\cdot, t)\|_{L^2(B_{r_0}(\bar{x}))} \leq \varepsilon \quad \text{in } [0,t_0],
		\]
		then we have 
		\[
		\left| P^{(1)}-P^{(2)} \right| + \left|M^{(1)} - M^{(2)}\right|    \leq C \varepsilon,  
		\]
	where $C$ is a positive constant depending on $M_0$, $m_0$, $d_0$,   $\lambda_0$, $\mu_0$, $\rho_0$ and $t_0$.
	\end{theorem}
	
	Uniqueness follows from Theorem \ref{t:stability} by letting $\varepsilon \to 0$. However, for the sake of the reader's understanding, in the sequel we first give the proof of uniqueness, since it clearly presents the role played by the elasto-gravitational coupling effect in solving the inverse problem. Also, focusing first on uniqueness first will make the proof of the desired Lipschitz stability a bit lighter, since we will reuse some of the calculations.
	
	Our uniqueness result can be summarized as follows:
	
	\begin{theorem}[uniqueness]\label{t:uniqueness} 	Let $B_{r_0}(\bar{x}) \subset \mathbb{R}^3 \setminus \overline{\Omega}$, $t_0>0$. Under the same hypothesis of Theorem \ref{t:stability} for $\textup{\textbf{f}}^{(1)}$, $\textup{\textbf{f}}^{(2)}$, $(\textup{\textbf{u}}^{(1)}, S^{(1)})$, $(\textup{\textbf{u}}^{(2)}, S^{(2)})$, $\lambda$,  $\mu$ and $\rho$, if   
		\[\nabla S^{(1)} = \nabla S^{(2)} \quad \text{in } B_{r_0}(\bar{x}) \times [0,t_0],
		\]
		then 
		\[M^{(1)} = M^{(2)} \; \text{ and } \; P^{(1)}=P^{(2)}. 
		\]
	\end{theorem}

	\subsection{Proof of Theorem \ref{t:uniqueness}}
	
	\begin{proof}[\unskip\nopunct]
		Define
		\[\textbf{u}:=\textbf{u}^{(2)}-\textbf{u}^{(1)}, \quad S:=S^{(2)}-S^{(1)}, \quad \textbf{f}:= \textbf{f}^{(2)}-\textbf{f}^{(1)}.\] 
		Trivially, ($\textbf{u}$, $S$) solves (in a weak sense) the following system of elasto-gravitational equations: 
		\begin{equation} 
			\begin{cases}
				\rho_0 \textbf{u}_{tt} - \text{div} (\mathbb{C}\nabla \textbf{u})  = \textbf{f} & (x,t) \in \Omega \times [0,\infty),\\
				\Delta S^- = -\rho_0\text{div}(\textbf{u}) & (x,t) \in \Omega \times [0,\infty),\\
				\Delta S^+ = 0 & (x,t) \in (\mathbb{R}^3 \setminus \overline{\Omega})\times [0,\infty), \\
				(\mathbb{C}\nabla \textbf{u}) \cdot \nu = 0 & (x,t) \in \partial \Omega\times [0,\infty), \\
				S^{-} = S^{+} & (x,t) \in \partial \Omega\times [0,\infty), \\
				(\nabla S^{-} + \textbf{u}) \cdot \nu  = \nabla S^{+} \cdot \nu & (x,t) \in \partial \Omega\times [0,\infty), \\
				\textbf{u}(x,0) =  \textbf{u}_t(x,0) = 0 & x \in \Omega,\\
				S^+ \to 0 & |x|\to \infty.
			\end{cases}
			\label{eq:elasto-gra-uni}
		\end{equation} 
		Also, since $\textup{\textbf{f}} \in H^1(\Omega, \mathbb{R}^3)$, the regularity results given in Section \ref{sec:forward} hold true. In particular, the estimate of Proposition \ref{p:energy} and its subsequent corollaries can be applied to \textbf{u}. 
		
		Let $\tau_0$ be the same as in Corollary \ref{c:null-cone}. In what follows, we assume that $t_0 <\tau_0$. We can do this, since our intention is to prove the uniqueness result for small times. 
		
		We begin by recalling that, from Corollary \ref{c:neigh}, $\textbf{u}=0$ in $\Omega \setminus \overline{\Omega}_{\frac{d_0}{2} - \frac{t_0}{\alpha_0}} \times [0,t_0]$ (cf. Figure \ref{fig:corollaries}). We have that
		\begin{equation*}
			\begin{cases}
				\Delta S^- = - \rho_0\text{div}(\textbf{u}) & (x,t) \in {\Omega}_{\frac{d_0}{2} - \frac{t_0}{\alpha_0}} \times [0,t_0],\\
				\Delta S^- = 0 & (x,t) \in \Omega \setminus \overline{\Omega}_{\frac{d_0}{2} - \frac{t_0}{\alpha_0}} \times [0,t_0],\\
				\Delta S^+ = 0 & (x,t) \in (\mathbb{R}^3 \setminus \overline{\Omega}) \times [0,t_0], \\
				S^-=S^+ & (x,t) \in \partial \Omega \times [0,t_0], \\ 
				\nabla S^- \cdot \nu = \nabla S^+ \cdot \nu  & (x,t) \in \partial \Omega \times [0,t_0],\\
				S^+ \to 0 & |x|\to \infty.
			\end{cases}
		\end{equation*}
		The continuity of the transmission conditions on $\partial \Omega$ assures that $S$ defined as
		\[S(x,t) = 
		\begin{cases} 
			S^-(x,t) & (x,t) \in \Omega \times [0,t_0],\\
			S^+(x,t) & (x,t) \in \mathbb{R}^3 \setminus \overline{\Omega} \times [0,t_0],
		\end{cases}  
		\]
		belongs to $H^2_\text{loc}\left(\mathbb{R}^3\setminus \overline{\Omega}_{\frac{d_0}{2} - \frac{t_0}{\alpha_0}}\right)$ for every $t \in [0,t_0]$. Now we apply the unique continuation property. Since $\nabla S=0$ in $B_{r_0}(\bar{x}) \times [0,t_0]$ and $S\to 0$ as $|x|\to\infty$, we have that
		\begin{equation}
			S(x,t)=0 \quad \text{in } \mathbb{R}^3 \setminus \overline{\Omega}_{\frac{d_0}{2} - \frac{t_0}{\alpha_0}} \quad \text{for every }t\in [0,t_0].
			\label{eq:ucp}
		\end{equation}
		
		Now the hypothesis that $\lambda=\lambda_0$, $\mu=\mu_0$ and $\rho=\rho_0$ are constants comes into play. Denote by $\phi$ a smooth function such that $\Delta \phi = 0$ in $\Omega$. Recall that $\textbf{u}$ solves
		\[
		\begin{cases}
			\rho_0 \textbf{u}_{tt} - \text{div}(\mathbb{C}\nabla \textbf{u}) = \textbf{f} & (x,t) \in \Omega \times [0,t_0],\\
			(\mathbb{C}\nabla \textbf{u}) \cdot \nu = 0 & (x,t) \in \partial \Omega \times [0,t_0],\\
			\textbf{u}(x,0)=\textbf{u}_t(x,0) =0 & x \in \Omega.
		\end{cases}
		\label{eq:elastic-tau}
		\]
		Multiplying the first equation of the system above by $\nabla \phi$ and integrating over $\Omega$ yield
		\begin{equation}
			\rho_0 \int_{\Omega} \textbf{u}_{tt} \cdot \nabla \phi + \int_{ \Omega} \partial_j (C_{ijkl} \partial_k u_l) \partial_i \phi = \int_{\Omega} \textbf{f} \cdot \nabla \phi.
			\label{eq:test1}
		\end{equation}
		We notice that
		\[
		\int_{ \Omega} \partial_j (C_{ijkl} \partial_k u_l) \partial_i \phi = \int_{\Omega} [\partial_j (C_{ijkl} \partial_k u_l \partial_i \phi) - C_{ijkl} \partial_k u_l \partial_{ji}^2 \phi] = - \int_{\Omega} C_{ijkl} \partial_k u_l \partial_{ji}^2 \phi, 
		\]
		where the last equality follows from the Neumann boundary condition for $\textbf{u}$. Since $\Delta \phi=0$, we have
		\[
		\begin{aligned} 
			\int_{\Omega} C_{ijkl} \partial_k u_l \partial_{ji}^2 \phi & = \lambda_0 \int_{\Omega} \Delta \phi \, \text{div}\textbf{u} + 2 \mu_0 \int_{ \Omega} \partial_i u_j \partial_{ij}^2 \phi \\ & = 2 \mu_0 \int_{\Omega} \partial_i (u_j \partial^2_{ij} \phi ) - 2\mu_0 \int_{ \Omega} u_j \partial_j \Delta \phi \\
			& = 2 \mu_0 \int_{\partial \Omega} u_j \partial_{ij}^2 \phi \nu_i = 0,  
		\end{aligned} 
		\]
		where the last equality follows from the fact that $\textbf{u}=0$ on $\partial \Omega \times [0,t_0]$. Thus equation \eqref{eq:test1} becomes
		\begin{equation}
			\rho_0 \int_{\Omega} \textbf{u}_{tt} \cdot \nabla \phi = \int_{\Omega} \textbf{f} \cdot \nabla \phi.
			\label{eq:test2}
		\end{equation}
		Define 
		\begin{equation}
			z := \rho_0 \int_{\Omega} \textbf{u} \cdot \nabla \phi. 
			\label{def:z}
		\end{equation}
		From \cite[Section 5.9.2]{evans1998partial} we have
		\begin{equation}
			z_{tt} = \rho_0  \int_{\Omega} \textbf{u}_{tt} \cdot \nabla \phi.
			\label{def:z-tt}
		\end{equation}
		To prove the uniqueness of both the moment tensor and location of the source, we first need to verify that $z(t)=0$ in $[0,t_0]$. Since $\textbf{u}=0$ on $\partial \Omega \times [0,t_0]$, integrating by parts yields:
		\begin{equation}
			\begin{aligned}
				z &= \rho_0 \int_{\Omega} \partial_j (u_j \phi) - \text{div} \textbf{u} \, \phi \\ 
				& = \rho_0 \int_{\partial \Omega} u_j \nu_j \phi + \rho_0 \int_{\Omega} \Delta S \phi \\
				& = \rho_0 \int_{\partial \Omega} \left[\frac{\partial S}{\partial \nu} \phi - S \frac{\partial \phi}{\partial \nu}\right] + \rho_0 \int_{\Omega} S \Delta \phi\\
				&= \rho_0 \int_{\partial \Omega} \left[\frac{\partial S}{\partial \nu} \phi - S \frac{\partial \phi}{\partial \nu}\right].
			\end{aligned}
			\label{eq:z-s} 
		\end{equation}
		Then the fact that $z=0$ follows from \eqref{eq:ucp}.
		Thus equation \eqref{eq:test2} gives
		\[
		\int_{ \Omega} \textbf{f} \cdot \nabla \phi = 0,
		\]
		for every $\phi$ harmonic function in $\Omega$. 
		
		Since $\textbf{f}=\textbf{f}^{(2)}-\textbf{f}^{(1)}$, the equality above yields:
		\begin{equation}
			\int_{B_{\frac{d_0}{2}}(P^{(1)})} \frac{q' (|x-P^{(1)}|)}{|x-P^{(1)}|} M^{(1)}(x-P^{(1)}) \cdot \nabla \phi = \int_{B_{\frac{d_0}{2}}(P^{(2)})} \frac{q' (|x-P^{(2)}|)}{|x-P^{(2)}|} M^{(2)}(x-P^{(2)}) \cdot \nabla \phi.
			\label{eq:test3}
		\end{equation}
		To prove that $M^{(1)} = M^{(2)}$ and $P^{(1)}=P^{(2)}$, we have to make a specific choice for $\phi$. 
		
		Before going further, we shall need the following technical result: 
		\begin{lemma}\label{lem:q}
			We have 
			\begin{equation}
				\int_{B_{\frac{d_0}{2}}} \frac{q' (|x|)}{|x|} x_j^2 = -1,
				\label{eq:q-even}
			\end{equation}
			and
			\begin{equation}
				\int_{B_{\frac{d_0}{2}}} \frac{q' (|x|)}{|x|} x_j = \int_{B_{\frac{d_0}{2}}} \frac{q' (|x|)}{|x|} x_j x_k x_\ell = 0.
				\label{eq:q-odd}
			\end{equation}
		\end{lemma}
		
		\begin{proof}
			We first prove \eqref{eq:q-even}. Define
			\[
			L_k := \int_{B_{\frac{d_0}{2}}} \frac{q' (|x|)}{|x|} x_j^2 .
			\]
			By symmetry,
			\[L:=L_1=L_2=L_3.
			\]
			We have
			\begin{equation*} 
				\begin{aligned}
					3 L = L_1+L_2+L_3 & = \int_{B_{\frac{d_0}{2}}} \frac{q' (|x|)}{|x|} (x_1^2+x_2^2+x_3^2) = \int_{B_{\frac{d_0}{2}}} q' (|x|) |x|\\
					& = 4 \pi \int_0^{\frac{d_0}{2}} q'(s) s^3   = -12 \pi \int_0^{\frac{d_0}{2}} q(s) s^2\\
					& = -3 \int_{B_{\frac{d_0}{2}}} q (|x|) \\
					& = -3,
				\end{aligned}
			\end{equation*}
			which implies $L=-1$.
			
			To prove \eqref{eq:q-odd}, it suffices to notice that we integrate an odd function over a spherically symmetric domain around the origin, hence the result is zero.
			
		\end{proof}
		
		We are now in the position of proving our uniqueness result. 
		
		\vspace{3mm}
		
		\textbf{Step 1: Moment tensor.} We first consider $\phi(x)=x_1x_2$. Trivially, such a function is harmonic in $\Omega$. Also, after a change of variables, we get 
		\begin{equation*}\begin{aligned} 
				\int_{B_{\frac{d_0}{2}}(P^{(1)})}  \frac{q'(|x-P^{(1)}|)}{|x-P^{(1)}|}  M^{(1)}_{kl}(x_l-(P^{(1)})_l) \partial_k \phi =  \int_{B_{\frac{d_0}{2}}} & \frac{q' (|y|)}{|y|}  \left(M_{1l}^{(1)} y_l y_2 + M_{2l}^{(1)} y_l y_1 \right.\\ 
				& \left. + \, M^{(1)}_{1l}y_l (P^{(1)})_2 +  M^{(1)}_{2l}y_l (P^{(1)})_1 \right). 
			\end{aligned}
		\end{equation*}
	Since $M^{(1)}$ is symmetric, Lemma \ref{lem:q} gives
	\begin{equation*}\begin{aligned} 
			\int_{B_{\frac{d_0}{2}}(P^{(1)})} \frac{q' (|x-P^{(1)}|)}{|x-P^{(1)}|} M^{(1)}_{kl}(x_l-(P^{(1)})_l) \partial_k \phi =  M_{12}^{(1)}\int_{B_{\frac{d_0}{2}}} & \frac{q' (|y|)}{|y|} \left(y_1^2 + y_2^2 \right) = -2 M_{12}^{(1)} . 
		\end{aligned}
	\end{equation*}
	Repeating the same calculations for $M^{(2)}$ yields
	\[
	\int_{B_{\frac{d_0}{2}}(P^{(2)})} \frac{q' (|x-P^{(2)}|)}{|x-P^{(2)}|} M^{(2)}(x-P^{(2)}) \cdot \nabla \phi = -2 M_{12}^{(2)}. 
	\]
	From \eqref{eq:test3}, we have 
	\[
	M_{12}^{(1)} = M_{12}^{(2)}.
	\]
	By taking $\phi = x_ix_j$, with $i\neq j$, we finally get
	\[
	M_{ij}^{(1)} = M_{ij}^{(2)}, \quad i, j=1,2,3.
	\]
	We now repeat the calculations above for $\phi = (x_1^2-x_2^2)/2$, and exploit the fact that both $M^{(1)}$ and $M^{(2)}$ have vanishing traces. After a change of variables, we get 
	\begin{equation*}\begin{aligned} 
			\int_{B_{\frac{d_0}{2}}(P^{(1)})}  \frac{q'(|x-P^{(1)}|)}{|x-P^{(1)}|}  M^{(1)}_{kl}(x_l-(P^{(1)})_l) \partial_k \phi =  \int_{B_{\frac{d_0}{2}}} & \frac{q' (|y|)}{|y|}  \left(M_{1l}^{(1)} y_l y_1 - M_{2l}^{(1)} y_l y_2 \right.\\ 
			& \left. + \, M^{(1)}_{1l}y_l (P^{(1)})_2 -  M^{(1)}_{2l}y_l (P^{(1)})_1 \right). 
		\end{aligned}
	\end{equation*}
	Thus
	\begin{equation*}\begin{aligned} 
			\int_{B_{\frac{d_0}{2}}(P^{(1)})}  \frac{q' (|x-P^{(1)}|)}{|x-P^{(1)}|} M^{(1)}_{kl}(x_l-(P^{(1)})_l) \partial_k \phi & =  \int_{B_{\frac{d_0}{2}}} \frac{q' (|y|)}{|y|} \left(M_{11}^{(1)} y_1^2 - M_{22}^{(1)} y_2^2 \right)\\ & = M_{22}^{(1)} - M_{11}^{(1)}.
		\end{aligned}
	\end{equation*}
	Repeating the same calculations for $M^{(2)}$ yields
	\[
	\int_{B_{\frac{d_0}{2}}(P^{(2)})} \frac{q' (|x-P^{(2)}|)}{|x-P^{(2)}|} M^{(2)}(x-P^{(2)}) \cdot \nabla \phi = M_{22}^{(2)} - M_{11}^{(2)}. 
	\]
	From \eqref{eq:test3}, we have 
	\[
	M_{11}^{(1)} - M_{11}^{(2)} =M_{22}^{(1)} - M_{22}^{(2)}.
	\]
	By taking $\phi = (x_1^2-x_3^2)/2$ and  $\phi=(x_2^2-x_3^2)/2$, we get also
	\[
	M_{11}^{(1)} - M_{11}^{(2)} =M_{33}^{(1)} - M_{33}^{(2)},
	\]
	\[
	M_{22}^{(1)} - M_{22}^{(2)} =M_{33}^{(1)} - M_{33}^{(2)}.
	\]
	Using the results above, together with tr$M^{(1)}=\text{tr}M^{(2)}=0$:
	\[
	\left(M_{11}^{(1)} - M_{11}^{(2)}\right) + \left(M_{22}^{(1)} - M_{22}^{(2)}\right) + \left(M_{33}^{(1)} - M_{33}^{(2)}\right)=0,
	\]
	we find
	\[
	M_{11}^{(1)} = M_{11}^{(2)}, \quad M_{22}^{(1)} = M_{22}^{(2)}, \quad M_{33}^{(1)} = M_{33}^{(2)}.
	\] 
	Putting everything together, we finally get $M^{(1)}=M^{(2)}$. 
	\vspace{3mm}
	
	\textbf{Step 2: Location of the source.} Denote now by $M$ the moment tensor. We begin by noticing that \eqref{eq:test3} can be rewritten as follows:
	\begin{equation}
		\int_{B_{\frac{d_0}{2}}(P^{(1)})} \frac{q' (|x-P^{(1)}|)}{|x-P^{(1)}|} M(x-P^{(1)}) \cdot \nabla \phi = \int_{B_{\frac{d_0}{2}}(P^{(2)})} \frac{q' (|x-P^{(2)}|)}{|x-P^{(2)}|} M(x-P^{(2)}) \cdot \nabla \phi.
		\label{eq:test4}
	\end{equation}
	We now consider $\phi = x_1^3-3x_2^2x_1$. Again, such a function is harmonic in $\Omega$. Also, after a change of variable, we get
	\begin{equation*}\begin{aligned} 
			 \int_{B_{\frac{d_0}{2}}(P^{(1)})} & \frac{q'(|x-P^{(1)}|)}{|x-P^{(1)}|}  M_{kl}(x_l-(P^{(1)})_l) \partial_k \phi  \\ &=  \int_{B_{\frac{d_0}{2}}} \frac{q' (|y|)}{|y|}  \left(M_{11} y_1 (3(y_1+(P^{(1)})_1)^2-3(y_2+(P^{(1)})_2)^2) +M_{12} y_2 (3(y_1+(P^{(1)})_1)^2  \right.\\ &  \quad \quad \quad \quad \left. - 3(y_2+(P^{(1)})_2)^2)  +M_{13} y_3 (3(y_1+(P^{(1)})_1)^2 - 3(y_2+(P^{(1)})_2)^2)  \right.\\ &  \quad \quad \quad \quad \left.+ \, M_{21}y_1 (-6(y_2+(P^{(1)})_2) (y_1+(P^{(1)})_1)) + \, M_{22}y_2 (-6(y_2+(P^{(1)})_2) (y_1+(P^{(1)})_1)) \right.\\ &  \quad \quad \quad \quad \left.+ \, M_{23}y_3 (-6(y_2+(P^{(1)})_2) (y_1+(P^{(1)})_1))\right). 
		\end{aligned}
	\end{equation*}
	Since $M_{12}=M_{21}$ and $M_{11}=-M_{22}-M_{22}$, Lemma \eqref{lem:q} implies that 
	\begin{equation*}\begin{aligned} 
			\int_{B_{\frac{d_0}{2}}(P^{(1)})}  & \frac{q'(|x-P^{(1)}|)}{|x-P^{(1)}|}  M_{kl}(x_l-(P^{(1)})_l) \partial_k \phi \\ &=  \int_{B_{\frac{d_0}{2}}} \frac{q' (|y|)}{|y|}  \left(6 M_{11} y_1^2 (P^{(1)})_1 - 6 M_{12} y_2^2 (P^{(1)})_2 - 6 M_{21} y_1^2 (P^{(1)})_2 - 6 M_{22} y_2^2(P^{(1)})_1)\right) \\
			& \quad \quad \quad \; = 6 (2M_{22}+M_{33}) (P^{(1)})_1 + 12 M_{12} (P^{(1)})_2. 
		\end{aligned}
	\end{equation*}
Repeating the same calculations for $P^{(2)}$ yields 
\begin{equation*}\begin{aligned} 
		\int_{B_{\frac{d_0}{2}}(P^{(2)})}  \frac{q'(|x-P^{(2)}|)}{|x-P^{(2)}|}  M_{kl}(x_l-(P^{(2)})_l) \partial_k \phi = 6 (2M_{22}+M_{33}) (P^{(2)})_1 + 12 M_{12} (P^{(2)})_2. 
	\end{aligned}
\end{equation*}
From \eqref{eq:test4}, we have 
\begin{equation}
	(2M_{22}+M_{33})((P^{(2)})_1 - (P^{(1)})_1) + 2M_{12}((P^{(2)})_2 - (P^{(1)})_2) = 0.
	\label{eq:p1}
\end{equation}
We now consider $\phi = x_2^3-3x_1^2x_2$. For such a choice of $\phi$, we obtain 
\begin{equation}
	(2M_{22}+M_{33})((P^{(2)})_2 - (P^{(1)})_2) + 2M_{12}((P^{(1)})_1 - (P^{(2)})_1) = 0.
	\label{eq:p2}
\end{equation}
We multiply \eqref{eq:p1} by $2M_{22}+M_{33}$: 
\[
(2M_{22}+M_{33})^2 ((P^{(1)})_1 - (P^{(2)})_1) + 2M_{12} (2M_{22} +M_{33}) ((P^{(1)})_2 - (P^{(2)})_2) =0
\]
and \eqref{eq:p2} by
$2M_{12}$:
\[
2M_{12}(2M_{22}+M_{33}) ((P^{(2)})_2 - (P^{(1)})_2) + 4M^2_{12} ((P^{(1)})_1 - (P^{(2)})_1) =0
\]
After summing both equations, we get
\begin{equation}
	((2M_{22} + M_{33})^2 + 4M_{12}^2)((P^{(1)})_1 - (P^{(2)})_1) = 0.
	\label{eq:cond1}
\end{equation}
We now consider $\phi = x_1^3-3x_3^2x_1$ and  $\phi = x_3^2-3x_3x_1^2$. We follow the same procedure  as the one described above. We get
\begin{equation}
	((M_{22} + 2M_{33})^2 + 4M_{31}^2)((P^{(1)})_1 - (P^{(2)})_1) = 0.
	\label{eq:cond2}
\end{equation}
Finally, we consider $\phi = x_1x_2x_3$. We get 
\begin{equation}
	M_{32} ((P^{(1)})_1 - (P^{(2)})_1)= 0.
	\label{eq:cond3}
\end{equation}
Assume now that $(P^{(1)})_1 \neq (P^{(2)})_1$. Then
\[
M_{32} = M_{12} = M_{31} =  M_{22}=M_{33} =0.
\]
Since $M$ is symmetric and tr$M=0$, the assumption $(P^{(1)})_1 \neq (P^{(2)})_1$ gives $M=0$, which is absurd.

Finally, to prove that $(P^{(1)})_2= (P^{(2)})_2$ and $(P^{(1)})_3= (P^{(2)})_3$, it will suffice to consider also the following harmonic functions:
\[
\phi = x_2^3 - 3x_2 x_3^2, \quad \text{and }\phi= x_3^3-3x_3x_2^2.  
\]
\end{proof}

\subsection{Proof of Theorem \ref{t:stability}}

\begin{proof}[\unskip\nopunct]
Define $\textbf{u}$, $S$, and $\textbf{f}$
as at the beginning of the uniqueness proof of Theorem \ref{t:uniqueness}.
To prove the desired Lipschitz stability estimate, we first need to propagate the smallness of the data 
\[\|\nabla S(\cdot, t)\|_{L^2(B_{r_0}(\bar{x}))} \leq \varepsilon \quad \text{in } [0,t_0],
\] 
into the integral 
\[
\left|\int_{\Omega} \textbf{f} \cdot \nabla \phi\right|,
\] 
since $\textbf{f}$, by definition, encodes the information on both $|M^{(2)}-M^{(1)}|$ and $|P^{(2)}-P^{(1)}|$. 

Let $\tau_0$ be the same as in Corollary \ref{c:null-cone}. As in the uniqueness proof of Theorem \ref{t:uniqueness}, we consider, without loss of generality, $t_0 <\tau_0$. Moreover, we recall that
\[
z_{tt} = \int_{\Omega} \textbf{f} \cdot \nabla \phi,
\]
where $z$ is defined as 
\[
z = \rho_0 \int_{\Omega} \textbf{u} \cdot \nabla \phi.
\]
Throughout the proof, we shall fix $t_1 \leq \frac{t_0}{2}$. Since
\[
z(0) = \rho_0 \int_{\Omega} \textbf{u}(x,0) \cdot \nabla \phi = 0, \quad 
z_t(0) = \rho_0 \int_{\Omega} \textbf{u}_t(x,0) \cdot \nabla \phi = 0,
\]
we have
\begin{equation*}
	\begin{aligned} z(t_1) & = \rho_0 \int_0^{t_1} z_t(s)\, \text{d}s  = \rho_0 \int_0^{t_1} \int_0^{s} z_{tt}(\eta) \, \text{d}\eta \, \text{d}s \\ & = \rho_0 \int_0^{t_1}\left( \int_\eta^{t_1} \text{d}s\right)z_{tt}(\eta) \, \text{d}\eta = \rho_0 \int_0^{t_1}\left( t_1-\eta\right)z_{tt}(\eta) \, \text{d}\eta \\ &
		= \rho_0\frac{ t_1^2}{2} \int_{\Omega} \textbf{f} \cdot \nabla \phi, 
	\end{aligned} 
\end{equation*}
hence
\begin{equation}
	\left| \int_{\Omega} \textbf{f} \cdot \nabla \phi \right| \leq \frac{2}{t_1^2} \frac{|z(t_1)|}{\rho_0}. 
	\label{eq:f-omega1-omega2}
\end{equation}
It is now apparent that, in the first part of the proof, our main efforts shall be devoted to proving that the smallness of the data propagates into $|z(t_1)|$. 

From \eqref{eq:z-s}, we write
\begin{equation*}
	\begin{aligned} |z(t_1)| & = \rho_0 \left|\int_{\partial \Omega} \left[\frac{\partial S(\cdot, t_1)}{\partial \nu}\phi - S(\cdot, t_1) \frac{\partial \phi}{\partial \nu}\right] \right| \\& \leq \rho_0 \int_{\partial \Omega} \left|\frac{\partial S(\cdot, t_1)}{\partial \nu}\right| |\phi| + \rho_0 \left| \int_{\partial \Omega}S(\cdot, t_1)\frac{\partial \phi }{\partial \nu} \, \right|. 
	\end{aligned} 
\end{equation*} 
We first notice that
\begin{equation*}
	\int_{\partial \Omega} \left|\frac{\partial S(\cdot, t_1)}{\partial \nu}\right| |\phi| \leq \underbrace{ \|\nabla S(\cdot, t_1)\|_{L^\infty (\partial \Omega)}}_{\omega_1} \int_{\partial \Omega} |\phi|. 
\end{equation*}
Secondly, for $x_0 \in B_{R_0} \setminus \overline{\Omega}$, where $B_{R_0}$ is such that
\[ 
\overline{B_{r_0}(\bar{x}) \cup \Omega}  \subset B_{R_0},
\]
we have
\begin{equation*}
	\begin{aligned} 
		\left|\int_{\partial \Omega} S(\cdot,t_1)\frac{\partial \phi(x)}{\partial \nu}  \right| & \leq \left| \int_{\partial \Omega} (S(\cdot, t_1) - S(x_0,t_1)) \frac{\partial \phi}{\partial \nu} \right| \\& \leq \underbrace{ \| S(\cdot, t_1) - S(x_0,t_1)\|_{L^\infty (\partial \Omega)}}_{\omega_2} \int_{\partial \Omega} \left|\frac{\partial \phi}{\partial \nu}\right|. 
	\end{aligned} 
\end{equation*}
Putting everything together, we obtain
\begin{equation*}
	|z(t_1)| \leq \rho_0 \left(\omega_1 \int_{\partial \Omega} \left|\phi\right| + \omega_2 \int_{\partial \Omega} \left|\frac{\partial \phi}{\partial \nu}\right|\right) \leq \rho_0 |\partial \Omega| \left(\omega_1 + \omega_2 \right) \|\phi \|_{C^1(\partial \Omega)}.
\end{equation*}
We thus begin to quantify the smallness of $|z(t_1)|$ by estimating 
\[
\omega_1 := \|\nabla S(\cdot, t_1)\|_{L^\infty(\partial \Omega)}.
\] 
We shall define $d_1$ such that $S(\cdot, t_1)$ is harmonic in $B_{\frac{d_1}{4}}(x)$, for $x \in \partial \Omega$. By Corollary \ref{c:neigh}, we set 
\[
d_1= \frac{d_0}{2} - \frac{t_1}{\alpha_0}.
\] 
For any $x \in \partial \Omega$, we have 
\begin{equation}
	\begin{aligned} 
		|\nabla S(x,t_1)|& \leq  \frac{1}{{\left|B_{\frac{d_1}{4}}(x)\right|}} \int_{B_{\frac{d_1}{4}}(x)} |\nabla S (\cdot, t_1)| \\ & \leq \frac{c}{d_1^{\frac{3}{2}}}\| \nabla S(\cdot, t_1) \|_{L^2 (B_{R_0} \setminus \overline{\Omega}_{\frac{d_1}{2}})}, 
		\label{eq:S-S}
	\end{aligned} 
\end{equation}
where $c= \left(\frac{48}{ \pi}\right)^{\frac{1}{2}}$. To estimate $\omega_1$, we exploit the following estimate of propagation of smallness  \cite[Theorem 5.1]{Alessandrini_2009}:
\begin{equation}
	\| \nabla S(\cdot, t_1) \|_{L^2 (B_{R_0} \setminus \overline{\Omega}_{\frac{d_1}{2}})} \leq C_0 \|\nabla S(\cdot, t_1)\|_{L^2(B_{r_0}(\bar{x}))}^\theta \|\nabla S(\cdot, t_1) \|_{L^2(B_{2 R_0} \setminus \overline{\Omega}_{d_1})}^{1-\theta}.
	\label{eq:three-spheres} 
\end{equation}
for some $\theta \in (0,1)$ and $C_0>0$.

We first estimate $\|\nabla S(\cdot, t_1) \|_{L^2(B_{2 R_0} \setminus \overline{\Omega}_{d_1})}$. Notice that 
\[
\text{div}(\nabla S(x, t_1))  = -\rho_0\text{div}\left(\textbf{u}(x, t_1) \rchi_{\Omega_{d_1}}(x)\right).
\]
From \cite[Section 5.9.1]{evans1998partial}, we obtain
\begin{equation}
	\begin{aligned}
		\|S(\cdot, t_1)\|_{H^1(\mathbb{R}^3)} \leq C \rho_0\|\text{div}\left(\textbf{u}(\cdot, t_1) \rchi_{\Omega_{d_1}}\right)\|_{H^{-1}(\mathbb{R}^3)} \leq C \rho_0 \left( \int_{\Omega_{d_1}} |\textbf{u}(\cdot, t_1)|^2\right)^{\frac{1}{2}}.	
	\end{aligned}
	\label{eq:S-u}
\end{equation}
We then use the energy estimate of Proposition \ref{p:energy-1} for $\textbf{u}(\cdot, t_1)$. We have
\begin{equation}
	\int_{\Omega} |\textbf{u}(\cdot, t_1)|^2  \leq \frac{t_1^3 e^{t_1/\rho_0}}{\rho_0} \int_{\Omega} |\textbf{f}|^2.
	\label{eq:u-f}
\end{equation}
We now focus on the right-hand side of the inequality above. We write
\begin{equation*}
	\textbf{f}(x) = - (M^{(2)} - M^{(1)}) q'(|x-P^{(2)}|) \frac{(x-P^{(2)})}{|x-P^{(2)}|} - M^{(1)} (\nabla (q(|x-P^{(2)}|)) - \nabla(q(|x-P^{(1)}|))).
\end{equation*}
Since $|M^{(1)}|,|M^{(2)}|\leq M_0$, we get
\begin{equation}
	\int_{\Omega} | \textbf{f}|^2 \leq 2 |M^{(2)}-M^{(1)}|^2 \int_{ \Omega} (q'(|x-P^{(2)}|))^2 + 2 {M}_0^2 \int_{ \Omega} \left|\nabla (q(|x-P^{(2)}|)) - \nabla(q(|x-P^{(1)}|))\right|^2.
	\label{eq:f-P-Q}
\end{equation}
Notice that
\[
\left|\partial_j (q(|x-P^{(2)}|)) - \partial_j(q(|x-P^{(1)}|))\right| = \left(\int_0^1 \left|\nabla \partial_j (q(|x-(P^{(1)} +(P^{(2)}-P^{(1)})\eta)|))\right| \, \text{d}\eta\right) |P^{(2)}-P^{(1)}|, 
\]
hence
\[
\int_{\Omega} \left|\partial_j (q(|x-P^{(2)}|)) - \partial_j(q(|x-P^{(1)}|))\right|^2 \! \!  \! \leq \!  |P^{(2)}-P^{(1)}|^2  \int_0^1 \! \! \int_\Omega \left|\nabla \partial_j (q(|x-(P^{(1)} +(P^{(2)}-P^{(1)})\eta)|))\right|^2.
\]
From \eqref{eq:f-P-Q}, we have
\begin{equation*}
	\left(\int_{\Omega} | \textbf{f}|^2 \right)^{\frac{1}{2}} \leq C_1 |M^{(2)}-M^{(1)}| + C_2 |P^{(2)}-P^{(1)}|,
\end{equation*}
where
\[
C_1 := \left( \int_{B_{\frac{d_0}{2}}} (q'(|x-P^{(2)}|))^2\right)^\frac{1}{2},
\]
\[
C_2:= \left( {M}_0^2 \int_0^1  \int_\Omega \left|\nabla \partial_j (q(|x-(P^{(1)} +(P^{(2)}-P^{(1)})\eta)|))\right|^2 \right)^\frac{1}{2}.
\]
From \eqref{eq:u-f}, we obtain
\begin{equation*}
	\left(\int_{\Omega} |\textbf{u}(\cdot, t_1)|^2 \right)^\frac{1}{2}\leq \sqrt{\frac{t_1^3 e^{t_1/\rho_0}}{\rho_0}} \left(C_1 |M^{(2)}-M^{(1)}| + C_2 |P^{(2)}-P^{(1)}|\right),
\end{equation*}
and, from \eqref{eq:S-u}, we have
\begin{equation*}
	\|\nabla S(\cdot, t_1) \|_{L^2(B_{2 R_0} \setminus \overline{\Omega}_{d_1})} \leq C_3 \left(|M^{(2)}-M^{(1)}| + |P^{(2)}-P^{(1)}|\right),
\end{equation*}
where 
\[
C_3 := C \sqrt{\rho_0 t_1^3 e^{t_1/\rho_0}} \max\{C_1,C_2\}.
\]
Estimate \eqref{eq:three-spheres} then gives
\begin{equation*}
	\| \nabla S(\cdot, t_1) \|_{L^2 (B_{R_0} \setminus \overline{\Omega}_{\frac{d_1}{2}})} \leq C_0 C_3 \varepsilon^\theta \left(|M^{(2)}-M^{(1)}| + |P^{(2)}-P^{(1)}|\right)^{1-\theta}.
\end{equation*}
Using \eqref{eq:S-S}, we finally get
\begin{equation}
	\omega_1 \leq \frac{c C_0 C_3}{d_1^{\frac{3}{2}}}  \varepsilon^\theta \left(|M^{(2)}-M^{(1)}| + |P^{(2)}-P^{(1)}|\right)^{1-\theta}.
	\label{eq:omega-1}
\end{equation}
In fact, the same inequality holds for
\begin{equation}
	\|\nabla S(\cdot, t_1)\|_{L^\infty (B_{R_0}\setminus {\Omega})} \leq \frac{c C_0 C_3}{d_1^{\frac{3}{2}}}  \varepsilon^\theta \left(|M^{(2)}-M^{(1)}| + |P^{(2)}-P^{(1)}|\right)^{1-\theta}. 
	\label{S-L-infty} 
\end{equation}

We now estimate 
\[
\omega_2 :=\| S(\cdot, t_1) - S(x_0,t_1)\|_{L^\infty (\partial \Omega)}, \quad \text{fixed some } x_0 \in B_{R_0} \setminus \overline{\Omega}. 
\] 
For any $x \in \partial \Omega$, we notice that 
\[
S(x, t_1) - S(x_0,t_1) = \int_0^{s_0} \frac{\text{d}}{\text{d} s} S(\gamma(s), t_1) = \int_0^{s_0} \nabla S(\gamma(s), t_1) \gamma'(s)  , 
\]
hence, by \eqref{S-L-infty},
\begin{equation*}
	\begin{aligned} 
		|S(x, t_1) - S(x_0,t_1)| & \leq   \|\nabla S(\cdot, t_1)\|_{L^\infty (B_{R_0}  \setminus \Omega)} \left|\int_0^{s_0} \gamma'(s)\right| \\ & \leq s_0\frac{c C_0 C_3}{d_1^{\frac{3}{2}}}  \varepsilon^\theta \left(|M^{(2)}-M^{(1)}| + |P^{(2)}-P^{(1)}|\right)^{1-\theta},
	\end{aligned}  
\end{equation*}
where $\gamma$ is an arc-length parameterized curve such that $\gamma(0) = x_0$ and $\gamma(s_0) = x$. Since $\Omega$ is convex, there exists $K>0$ depending only on the diameter of $\Omega$ and $R_0$ such that \[ 
s_0\leq K, \quad \forall x_0  \in B_{R_0} \setminus \overline{ \Omega}, \quad \forall x \in \partial \Omega,
\]
hence
\begin{equation}
	\omega_2 \leq K \frac{c C_0 C_3}{d_1^{\frac{3}{2}}}  \varepsilon^\theta \left(|M^{(2)}-M^{(1)}| + |P^{(2)}-P^{(1)}|\right)^{1-\theta}. 
	\label{eq:omega-2}
\end{equation}  
Finally, putting \eqref{eq:f-omega1-omega2}, \eqref{eq:omega-1} and 
\eqref{eq:omega-2} together, we find that
\begin{equation}
	\left|\int_{\Omega} \textbf{f} \cdot \nabla \phi \right| \leq  \varepsilon_1 \|\phi \|_{C^1(\partial \Omega)},
	\label{eq:f-epsilon}	
\end{equation} 
where
\begin{equation}
	\varepsilon_1 := \frac{2 C_4}{t_1^2} \varepsilon^\theta (|M^{(2)}-M^{(1)}|+|P^{(2)}-P^{(1)}|)^{1-\theta},
	\label{eq:varepsilon}
\end{equation}
and
\[
C_4:= \frac{c C_0 C_3 }{d_1^{\frac{3}{2}}} |\partial \Omega|  \max\left\{  1, K\right\}.
\]

We are now able to propagate the smallness of the data directly into $|M^{(2)}-M^{(1)}|$ and $|P^{(1)}-P^{(2)}|$ by using the definition of $\textbf{f}$: 
\begin{equation*}
	\int_{\Omega} \textbf{f} \cdot \nabla \phi = \int_{B_{\frac{d_0}{2}}(P^{(1)})} \frac{q' (|x-P^{(1)}|)}{|x-P^{(1)}|} M^{(1)}(x-P^{(1)}) \cdot \nabla \phi - \int_{B_{\frac{d_0}{2}}(P^{(2)})} \frac{q' (|x-P^{(2)}|)}{|x-P^{(2)}|} M^{(2)}(x-P^{(2)}) \cdot \nabla \phi.
\end{equation*} 
In what follows, we retrace some calculations done in the uniqueness proof of Theorem \ref{t:uniqueness}, starting from equation \eqref{eq:test3} up to \eqref{eq:cond3}. 

\vspace{3mm}

\textbf{Step 1: Moment tensor.} We begin by recalling that, for $\phi=x_ix_j$, $i\neq j$, we have:
\[
\int_{\Omega} \textbf{f} \cdot \nabla \phi = -2 M^{(2)}_{ij} + 2M^{(1)}_{ij}.
\]
By \eqref{eq:f-epsilon}, we thus find
\begin{equation}
	\left|M^{(2)}_{ij} - M^{(1)}_{ij}\right| \leq C_5 \frac{\varepsilon_1}{2}, \quad i,j=1,2,3, \quad i\neq j,
	\label{eq:Mij}
\end{equation}
where $C_5$ refers to a constant  greater than $\|\phi\|_{\partial \Omega}$ for any choices of $\phi$ we will make throughout the remainder of the proof. By taking $\phi = (x_1^2-x_2^2)/2$, $\phi = (x_1^2-x_3^2)/2$, and $\phi = (x_2^2-x_3^2)/2$, we also get
\begin{equation*}
	\left|(M^{(2)}_{11} - M^{(1)}_{11}) - (M^{(2)}_{22} - M^{(1)}_{22})\right| \leq C_5 \varepsilon_1,
\end{equation*}
\begin{equation*}
	\left|(M^{(2)}_{11} - M^{(1)}_{11}) - (M^{(2)}_{33} - M^{(1)}_{33})\right| \leq C_5 \varepsilon_1,
\end{equation*}
and
\begin{equation*}
	\left|(M^{(2)}_{22} - M^{(1)}_{22}) - (M^{(2)}_{33} - M^{(1)}_{33})\right| \leq C_5 \varepsilon_1.
\end{equation*}
We can write
\begin{equation*}
	M^{(2)}_{11} - M^{(1)}_{11} = M^{(2)}_{22} - M^{(1)}_{22} + \sigma_1,
\end{equation*}
\begin{equation*}
	M^{(2)}_{22} - M^{(1)}_{22} = M^{(2)}_{33} - M^{(1)}_{33} + \sigma_2,
\end{equation*}
and
\begin{equation*}
	M^{(2)}_{33} - M^{(1)}_{33} = M^{(2)}_{11} - M^{(1)}_{11} + \sigma_3,
\end{equation*}
where
\[
|\sigma_i| \leq C_5 \varepsilon_1, \quad i=1,2,3.
\]
Using the results above, together with tr$M^{(1)}=\text{tr}M^{(2)}=0$:
\[
\left(M_{11}^{(2)} - M_{11}^{(1)}\right) + \left(M_{22}^{(2)} - M_{22}^{(1)}\right) + \left(M_{33}^{(2)} - M_{33}^{(1)}\right)=0,
\]
we find
\begin{equation}
	|M_{ii}^{(2)} - M_{ii}^{(1)}| \leq C_5 \frac{\varepsilon_1}{3}, \quad i=1,2,3.
	\label{eq:Mii}
\end{equation} 
Putting \eqref{eq:Mij} and \eqref{eq:Mii} together, we finally get
\begin{equation}
	|M^{(2)} - M^{(1)}| \leq C_6 \varepsilon_1,
	\label{estimate-M}
\end{equation}
where 
\[
C_6 = \frac{5\sqrt{3}}{6}C_5.
\]

\vspace{3mm}

\textbf{Step 2: Location of the source.} We begin by noticing that estimate \eqref{estimate-M} implies
\[
M^{(2)} = M^{(1)} + M^{\sigma}, 
\]
where $M^{\sigma} \in \mathbb{M}^3$ and
\[
|M^{\sigma}| \leq C_6 \varepsilon_1.
\]
Repeating the same calculations done in the uniqueness proof of Theorem \ref{t:uniqueness} yields
\[
6(2M_{22}^{(1)}+ M_{33}^{(1)}) (P^{(1)} - P^{(2)})_1 + 12 M^{(1)}_{12} (P^{(1)} - P^{(2)})_2 = 6 (2 M^{\sigma}_{22} + M^{\sigma}_{33}) (P^{(2)})_1 + 12 M^{\sigma}_{12} (P^{(2)})_2 + \int_{ \Omega} \textbf{f} \cdot \nabla \phi 
\]
for $\phi=x_1^3-3x_2^2x_1$. Hence
\[
\left|(2M_{22}^{(1)}+ M_{33}^{(1)}) (P^{(1)} - P^{(2)})_1 + 2 M^{(1)}_{12} (P^{(1)} - P^{(2)})_2\right| \leq C_7 \varepsilon_1, 
\]
where $C_7$ depends on the diameter of $\Omega$. Analogously, by taking $\phi = x_2^3-3x_1^2x_2$, we obtain
\[
\left|(2M_{22}^{(1)}+ M_{33}^{(1)}) (P^{(1)} - P^{(2)})_2 + 2 M^{(1)}_{12} (P^{(1)} - P^{(2)})_1\right| \leq C_7 \varepsilon_1. 
\]
We can then write the following system:
\begin{equation*}
	\begin{cases}
		(2M_{22}^{(1)}+ M_{33}^{(1)}) (P^{(1)} - P^{(2)})_1 + 2 M^{(1)}_{12} (P^{(1)} - P^{(2)})_2 = \tilde{\sigma}_1, \vspace{2mm}\\
		2 M^{(1)}_{12} (P^{(1)} - P^{(2)})_2 - (2M_{22}^{(1)}+ M_{33}^{(1)}) (P^{(1)} - P^{(2)})_2  = \tilde{\sigma}_2,
	\end{cases}
\end{equation*}
where 
\[
|\tilde{\sigma}_1|, |\tilde{\sigma}_2| \leq C_7 \varepsilon_1.
\]
Calculating the determinant gives
\[
-\left( (2M^{(1)}_{22} + M^{(1)}_{33})^2 + 4 (M^{(1)}_{12})^2\right),
\]
hence
\begin{equation*}
	(P^{(1)} - P^{(2)})_1 = 
	\frac{\begin{vmatrix} \tilde{\sigma}_1 & 2M_{12}^{(1)}\\
			\tilde{\sigma}_2 & 2 M_{22}^{(1)} + M_{33}^{(1)} 
	\end{vmatrix}}{-\left( (2M^{(1)}_{22} + M^{(1)}_{33})^2 + 4 (M^{(1)}_{12})^2\right)}. 
	\label{eq:det1} 
\end{equation*}
We have
\begin{equation*}
	\left| \left( (2M^{(1)}_{22} + M^{(1)}_{33})^2 + 4 (M^{(1)}_{12})^2\right) (P^{(1)} - P^{(2)})_1 \right| \leq 5 C_7 M^2_0 \varepsilon_1
\end{equation*}
Repeating the same calculations done above for $\phi = x_1^3-3x_3^2x_1$, $\phi = x_3^3-3x_3x_1^2$ and $\phi = x_1 x_2 x_3$ yields
\begin{equation*}
	\left| \left( (M^{(1)}_{22} + 2M^{(1)}_{33})^2 + 4 (M^{(1)}_{13})^2\right) (P^{(1)} - P^{(2)})_1 \right| \leq 5C_7 M^2_0 \varepsilon_1,
\end{equation*}
and
\begin{equation*}
	\left| (M^{(1)}_{32})^2 (P^{(1)} - P^{(2)})_1 \right| \leq C_7 M^2_0 \varepsilon_1.
\end{equation*}
By putting everything together and using tr$M^{(1)} = 0$, we obtain:
\begin{equation}
	\begin{cases}
		\left| \left( (M^{(1)}_{11} - M^{(1)}_{22})^2 + 4 (M^{(1)}_{12})^2\right) (P^{(1)} - P^{(2)})_1 \right| \leq C_8 \varepsilon_1, \vspace{2mm}\\
		\left| \left( (M^{(1)}_{33} - M^{(1)}_{11})^2 + 4 (M^{(1)}_{13})^2\right) (P^{(1)} - P^{(2)})_1 \right| \leq C_8 \varepsilon_1, \vspace{2mm}\\ 
		\left| M^{(1)}_{32} (P^{(1)} - P^{(2)})_1 \right| \leq C_8 \varepsilon_1.
	\end{cases}
	\label{eq:system-P}
\end{equation}
where $C_8 = 5 C_7 M_0^2$. 
We set $\lambda >0$ such that
\begin{equation}
	\begin{cases} (M^{(1)}_{11} - M^{(1)}_{22})^2 + 4 (M^{(1)}_{12})^2 \leq \lambda^2, \vspace{2mm}\\
		(M^{(1)}_{33} - M^{(1)}_{11})^2 + 4 (M^{(1)}_{13})^2 \leq \lambda^2, \vspace{2mm} \\
		(M^{(1)}_{32})^2 \leq \lambda^2.
	\end{cases}
	\label{eq:lambda-sys}
\end{equation}
We first want to show that 
\begin{equation}
	|M_{ij}^{(1)}| \leq \lambda, \quad i,j=1,2,3.
	\label{eq:M-lambda}
\end{equation}
System \eqref{eq:lambda-sys} gives
\[
|M^{(1)}_{12}| \leq \frac{\lambda}{2}, \quad |M^{(1)}_{13}| \leq \frac{\lambda}{2}, \quad |M_{32}^{(1)}| \leq \lambda,
\]
and
\[
|M^{(1)}_{11} - M^{(1)}_{22}| \leq \lambda , \quad |M^{(1)}_{33} - M^{(1)}_{11}| \leq \lambda,
\]
that is
\[
M^{(1)}_{22} = - \tilde{\lambda}_1 + M^{(1)}_{11}, \quad  M^{(1)}_{33} = - \tilde{\lambda}_2 + M^{(1)}_{11},
\]
with $|\tilde{\lambda}_1|, |\tilde{\lambda}_2| \leq \lambda$. Using tr$M^{(1)}=0$, the equations above yield
\[
3 M_{11}^{(1)}- \tilde{\lambda}_1 - \tilde{\lambda}_2 =0,
\]
hence
\[
|M_{11}^{(1)}|\leq \frac{2}{3} \lambda. 
\]
Since the same calculations can be done for $M_{22}^{(1)}$ and $M_{33}^{(1)}$, we proved \eqref{eq:M-lambda}, hence
\[
|M|^2 \leq 9 \lambda^2.
\]
We recall now that \eqref{eq:M-bound} implies $|M|^2 \geq m_0^2$, hence
\[
\lambda \geq \frac{m_0}{3}.
\]
This means that if we set
\[
\lambda = \frac{1}{2} \frac{m_0}{3}
\]
in \eqref{eq:lambda-sys}, at least one of the coefficients of the system must be greater than  $\frac{1}{2} \frac{m_0}{3}$.
By \eqref{eq:system-P}, we have
\[
\left|(P^{(1)} - P^{(2)})_1 \right| \leq \frac{6C_8}{m_0} \varepsilon_1.
\]
By repeating the same calculations for $(P^{(1)} - P^{(2)})_2$, we finally obtain
\[
\left|P^{(1)} - P^{(2)} \right| \leq \frac{6C_8}{m_0} \varepsilon_1.
\]
Putting everything together yields
\begin{equation*}
	\left|P^{(1)} - P^{(2)} \right| + \left|M^{(1)} - M^{(2)} \right| \leq \frac{C_9}{m_0} \varepsilon_1,
\end{equation*}
where
\[
C_9 = m_0 C_6 + \frac{6C_8}{m_0}.
\]
We now recall from \eqref{eq:varepsilon} that
\[
\varepsilon_1 = \frac{2 C_4}{t_1^2} \varepsilon^\theta (|M^{(2)}-M^{(1)}|+|P^{(2)}-P^{(1)}|)^{1-\theta},
\]
hence
\begin{equation}
	\left|P^{(1)} - P^{(2)} \right| + \left|M^{(1)} - M^{(2)} \right| \leq \frac{C_{10}}{m_0} \varepsilon^\theta (|M^{(2)}-M^{(1)}|+|P^{(2)}-P^{(1)}|)^{1-\theta},
	\label{eq:estimate-P-M-fin}
\end{equation}
where
\[
C_{10} = \frac{2 C_4}{t_1^2} C_9.
\]
From \eqref{eq:estimate-P-M-fin}, we finally get the desired Lipschitz stability estimate
\[
\left|P^{(1)} - P^{(2)} \right| + \left|M^{(1)} - M^{(2)} \right| \leq \left(\frac{C_{10}}{m_0} \right)^{\frac{1}{\theta}}  \varepsilon.
\] 
\end{proof}

\section{Concluding remarks}
In this paper we have studied an early-warning inverse source problem for the elasto-gravitational equations that is motivated by seismology. The problem involves a mixed hyperbolic-elliptic system of partial differential equations describing elastic wave displacement and gravity perturbations produced by a source in a homogeneous bounded medium. Within the Cowling approximation, we have shown how to turn the so-called elasto-gravitational coupling to our advantage: the changes of the gravitational field are generated instantaneously by Poisson's equation and thus can be used to solve the early-warning inverse problem without having to impose a minimum on the time needed to determine the source. 

This paper is the first step towards developing the theoretical framework of the PEGS inverse problem. Proving uniqueness and stability theorems for the self-gravitating elastic equations will be the object of future research. 

\appendix

\section{Energy estimates} 
\label{sec:appA}

\subsection{Proof of proposition \ref{p:energy-1}}
\begin{proof}[\unskip\nopunct]
We begin by multiplying both sides of the first equation of the system \eqref{eq:elastic-dp} by $\textbf{u}_t$. We get
\[
\rho_0 \textbf{u}_{tt} \cdot \textbf{u}_t - \text{div}(\mathbb{C}\nabla \textbf{u}) \cdot \textbf{u}_t = \textbf{f} \cdot \textbf{u}_t,
\]
hence
\[
\frac{\rho_0}{2} \partial_t|\textbf{u}_t|^2 + \frac{1}{2} \partial_t (\mathbb{C} \nabla \textbf{u} \cdot \nabla \textbf{u}) - \text{div}(\mathbb{C} \nabla \textbf{u} \cdot \textbf{u}_t) = \textbf{f} \cdot \textbf{u}_t.
\]
We integrate the equation above over $\Omega \times [0,s]$. Since $\mathbb{C} \nabla \textbf{u} \cdot \nu =0$ on $\partial \Omega$, we obtain
\[
\frac{1}{2} \int_{\Omega} \left(\rho_0|\textbf{u}_t(x,s)|^2  + \mathbb{C} \nabla \textbf{u}(x,s) \cdot \nabla \textbf{u}(x,s)\right) \, \text{d}x = \int_0^s \int_{\Omega} \textbf{f}(x) \cdot \textbf{u}_t(x,t) \, \text{d}x  \text{d}t,
\]	
hence
\[
\int_{\Omega} |\textbf{u}_t(x,s)|^2 \, \text{d}x \leq \frac{1}{\rho_0} \left(\int_0^s \int_{\Omega} \left|\textbf{f}(x)\right|^2 \, \text{d}x  \text{d}t + \int_0^s \int_{\Omega} \left|\textbf{u}_t(x,t)\right|^2 \, \text{d}x  \text{d}t\right).
\]		
By Gr\"onwall's inequality, we find
\begin{equation}
	\int_{\Omega} |\textbf{u}_t(x,s)|^2 \, \text{d}x \leq \frac{1}{\rho_0} \left(\int_0^s \int_{\Omega} |\textbf{f}(x)|^2  \, \text{d}x\text{d}t\right) e^{s/\rho_0}.
	\label{eq:gronwall}
\end{equation}
Since $\textbf{u}(x,0)=0$, we have
\[
|\textbf{u}(x,\tau)| \leq \int_0^\tau |\textbf{u}_t(x, s)|\, \text{d}s \leq \sqrt{\tau} \left(\int_0^\tau |\textbf{u}_t(x,s)|^2 \, \text{d}s \right)^{\frac{1}{2}},
\]	
hence,
by \eqref{eq:gronwall}, we get
\begin{equation*}
	\begin{aligned}
		\int_{ \Omega}|\textbf{u}(x,\tau)|^2 \, \text{d}x & \leq \tau \int_0^\tau \int_{ \Omega} |\textbf{u}_t(x,s)|^2 \, \text{d}x \text{d}s \\ 
		& \leq \tau \int_0^\tau \frac{1}{\rho_0} \left(\int_0^s \int_{\Omega} |\textbf{f}(x)|^2  \, \text{d}x\text{d}t\right) e^{s/\rho_0} \, \text{d}s\\
		& \leq \frac{\tau e^{\tau/\rho_0}}{\rho_0} \int_0^\tau \left(\int_0^s \int_{\Omega} |\textbf{f}(x)|^2  \, \text{d}x\text{d}t\right) \, \text{d}s
	\end{aligned}	
\end{equation*}
hence
\begin{equation*}
	\int_{ \Omega}|\textbf{u}(x,\tau)|^2 \, \text{d}x \leq \frac{\tau^3 e^{\tau/\rho_0}}{\rho_0} \int_{\Omega} |\textbf{f}(x)|^2  \, \text{d}x.
\end{equation*}

\end{proof} 

\subsection{Proof of proposition \ref{p:energy}}

\begin{proof}[\unskip\nopunct]

Define $\textbf{v} := \textbf{u} \, e^{-\gamma t}$, $\gamma >0$. It is easy to check that
\[
\textbf{u}_t = (\textbf{v}_t + \gamma \textbf{v}) e^{\gamma t},
\]
\[
\textbf{u}_{tt} = (\textbf{v}_{tt} + 2 \gamma \textbf{v}_t + \gamma^2 \textbf{v}) e^{\gamma t},
\]
hence, if $\textbf{u}$ solves \eqref{eq:elastic-dp}, we have that $\textbf{v}$ solves
\[
\rho_0 \left(\textbf{v}_{tt} + 2 \gamma \textbf{v}_{t} + \gamma^2 \textbf{v}\right)-\text{div} (\mathbb{C} \nabla \textbf{v})=\textbf{f}\, e^{-\gamma t}.
\]
After multiplying both sides of the above equation by $\textbf{v}_t$, we obtain:
\[\frac{\rho_0}{2}\left( \partial_t |\textbf{v}_t|^2 +4 \gamma |\textbf{v}_t|^2 + {\gamma^2}\partial_t |\textbf{v}|^2\right) - \partial_j (C_{ijkl} \partial_k v_l) v_{i,t} =\textbf{f} \cdot  \textbf{v}_t \, e^{-\gamma t}.
\]
We integrate the equation above over ${\widetilde{K}}$. We write
\begin{equation}
	\begin{aligned} 
		\int_{\widetilde{K}  } \partial_t \left(\frac{\rho_0}{2} |\textbf{v}_t|^2 + \rho_0 \frac{\gamma^2}{2} |\textbf{v}|^2\right) -  \partial_j (C_{ijkl} \partial_k v_l) v_{i,t}  + 2\rho_0 \gamma\int_{\widetilde{K}} |\textbf{v}_t|^2 = \int_{\widetilde{K}} \textbf{f} \cdot  \textbf{v}_t \, e^{-\gamma t} .
		\label{eq:int-1}
	\end{aligned} 
\end{equation}
We want to prove that, if we set 
\[
I_1 := \int_{\widetilde{K}} \partial_t \left(\frac{\rho_0 }{2} |\textbf{v}_t|^2 + \rho_0 \frac{\gamma^2}{2} |\textbf{v}|^2\right) -  \partial_j (C_{ijkl} \partial_k v_l) v_{i,t},
\]
then $I_1\geq 0$ for a precise choice of the parameter $\alpha$. We first observe that
\[
\begin{aligned} 
	\int_{\widetilde{K}} \partial_j (C_{ijkl} \partial_k v_l) v_{i,t} & =  \int_{\widetilde{K}} \partial_j (C_{ijkl} \partial_k v_l v_{i,t}) - C_{ijkl} \partial_k v_l \partial_j v_{i,t}\\ & = \int_{\widetilde{K}} \partial_j (C_{ijkl} \partial_k v_l v_{i,t}) - \frac{1}{2} \partial_t (C_{ijkl} \partial_k v_l \partial_j  v_i).
\end{aligned} 
\]
Let us set 
\[
\Gamma_1 := (\Omega \times \{0\}) \cap K, \quad 
\Gamma_2 := (\Omega \times (0,\infty)) \cap \partial K, \quad 
\Gamma_3 := (\partial \Omega \times (0,\infty)) \cap K.
\]
Notice that
\[\partial \widetilde{K} = 
\Gamma_1 \cup \Gamma_2 \cup \Gamma_3.\] 
Thanks to the regularity provided by Theorem \ref{t:regularity}, we employ the divergence theorem and obtain 
\[
\int_{\partial \widetilde{K}} \left(\frac{\rho_0}{2} |\textbf{v}_t|^2 + \rho_0 \frac{\gamma^2}{2} |\textbf{v}|^2 + \frac{1}{2} C_{ijkl} \partial_k v_l \partial_j  v_i\right) (\textbf{N}\cdot \textbf{e}_t)-  \left(C_{ijkl} \partial_k v_l v_{i,t}  \right)(\textbf{N}\cdot \textbf{e}_j) ,
\]
where $\textbf{N}$ is the outward normal vector to $\partial{\widetilde{K}}$. It is defined as follows:
\begin{itemize}
	\item[(i)] on $\Gamma_1$: $\textbf{N}=(0,0,0,-1)$;
	\item[(ii)] on $\Gamma_2$: $\textbf{N} =\frac{1}{\sqrt{1+\alpha^2}}\left(\frac{\alpha (x-x_0)}{|x-x_0|},1\right)$;
	\item[(iii)] on $\Gamma_3$: $\textbf{N}=(\nu,0)$, where $\nu$ is the outward normal vector to $\partial \Omega$.
\end{itemize}
Since $\textbf{v}= \textbf{v}_t =0$ on $\Gamma_1$ and $(\mathbb{C} \nabla \textbf{v}) \cdot \nu =0$ on $\Gamma_3$, we can write
\[
\begin{aligned}
	I_1=\frac{1}{2\sqrt{1+\alpha^2}} \int_{\Gamma_2}  \underbrace{\rho_0 |\textbf{v}_t|^2 + \rho_0 \gamma^2 |\textbf{v}|^2 + C_{ijkl} \partial_k v_l \partial_j v_i - 2 \alpha \left(C_{ijkl} \partial_k v_l v_{i,t}  \right) \frac{(x-x_0)_j}{|x-x_0|}}_{\mathlarger{I_2}}.
\end{aligned}
\] 
We now define
\[
\xi:= \textbf{v}_t,  \quad  \quad {A} := \frac{\nabla \textbf{v} + (\nabla \textbf{v})^\top}{2}, \quad \eta:=\frac{(x-x_0)}{|x-x_0|}.
\]
Since 
\[\frac{\rho_0  \gamma^2}{2\sqrt{1 + \alpha^2}} \int_{\Gamma_2} |\textbf{v}|^2\geq 0,
\] 
we have
\[
I_2 \geq \rho_0 |\xi|^2 + \lambda_0 (\text{tr} {A})^2 + 2 \mu_0 |{A}|^2 - 2\alpha (\lambda (\text{tr} {A}) \xi \cdot \eta + 2 \mu_0 {A} \xi \cdot \eta). 
\]
Since $|\eta|=1$ and \[
\rho_0>0, \quad \lambda_0>0, \quad \mu_0 > 0,
\]
we have, for every $\epsilon >0$,
\[
\begin{aligned} 
	I_2 & \geq \rho_0 |\xi|^2 + \lambda_0 (\text{tr}{A})^2 + 2\mu_0 |{A}|^2 -2\alpha \left(\lambda_0 \left( \frac{1}{2\epsilon} (\text{tr}{A})^2 + \frac{\epsilon}{2} |\xi|^2\right) + 2 \mu_0 \left(\frac{1}{2\epsilon} |{A}|^2 + \frac{\epsilon}{2} |\xi|^2\right)\right)\\ 
	& \geq  |\xi|^2 (\rho_0 - \lambda_0 \epsilon \alpha - 2\mu_0 \epsilon \alpha ) + (\text{tr} {A})^2 \left(\lambda_0 - \frac{\alpha \lambda_0}{\epsilon}\right) + |{A}|^2 \left(2\mu_0 - \frac{2\alpha \mu_0}{\epsilon}\right)\\
	& \geq |\xi|^2 (\rho_0 -  \epsilon \alpha (\lambda_0+ 2\mu_0) ) + \left(1 - \frac{\alpha}{\epsilon}\right) (\lambda_0 (\text{tr}{A})^2 + 2\mu_0 |{A}|^2).
\end{aligned} 
\]
Finally, for $\alpha=\alpha_0$ and $\epsilon=\alpha_0$, we obtain
\[
I_2 \geq |\xi|^2 (\rho_0 - \alpha^2_0 (\lambda_0 + 2\mu_0) ) \geq 0,
\]
which implies that 
\[
I_1 \geq 0.
\]
This means that, by \eqref{eq:int-1}, we have 
\begin{equation}
	2\rho_0 \gamma \int_{\widetilde{K}}  |\textbf{v}_t|^2 \leq I_1 + 2\rho_0 \gamma \int_{\widetilde{K}} |\textbf{v}_t|^2 =\int_{\widetilde{K}} \textbf{f} \cdot \textbf{v}_t \, e^{-\gamma t}. 
	\label{eq:ineq-1}
\end{equation}
By the Cauchy-Schwarz inequality, we get
\[
\int_{\widetilde{K}} \textbf{f} \cdot \textbf{v}_t \, e^{-\gamma t} \leq \left(\int_{\widetilde{K}} |\textbf{f}|^2 e^{-2\gamma t}\right)^{\frac{1}{2}} \left(\int_{\widetilde{K}} |\textbf{v}_t|^2\right)^{\frac{1}{2}},
\]
hence \eqref{eq:ineq-1} becomes:
\begin{equation}
	(2\rho_0 \gamma)^2 \int_{\widetilde{K}} |\textbf{v}_t|^2 \leq \int_{\widetilde{K}} |\textbf{f}|^2 e^{-2\gamma t}.
	\label{eq:ineq-2}
\end{equation}
Observe that, since $\textbf{v}(x,0)=0$ and using
again the Cauchy-Schwarz inequality, we have
\[
|\textbf{v}(x,t)|^2 \leq t \int_0^t |\textbf{v}_s (x,s)|^2 \, \text{d}s \leq \tau \int_0^t |\textbf{v}_s (x,s)|^2 \, \text{d}s
\]
since $t \leq \tau$.
Notice that  
\[
\widetilde{K} = \left\{ (x,t) \in \mathbb{R}^4 \text{ such that } x\in B_{\tau/\alpha_0}(x_0) \cap \Omega, \; 0 \leq t \leq \tau - \alpha_0 |x-x_0| \right\},
\]
hence
\[
\int_{\widetilde{K}} |\textbf{v}(x,t)|^2\, \text{d}x\text{d}  t\leq \tau \int_{B_{\tau/\alpha_0}(x_0) \cap \Omega} \left( \int_{0}^{\tau-\alpha_0|x-x_0|} \left( \int_0^t |\textbf{v}_s (x,s)|^2 \, \text{d}s\right) \, \text{d}t \right)  \text{d}x. 
\]
Since $t \leq \tau-\alpha_0|x-x_0|$, we write 
\begin{equation}
	\begin{aligned} 
		\int_{\widetilde{K}} |\textbf{v}(x,t)|^2\, \text{d}x\text{d}  t& \leq \tau \int_{B_{\tau/\alpha_0}(x_0) \cap \Omega} (\tau-\alpha_0|x-x_0|)\left( \int_{0}^{\tau-\alpha_0|x-x_0|} |\textbf{v}_s (x,s)|^2 \, \text{d}s \right)  \text{d}x\\
		& 
		\leq \tau^2 \int_{\widetilde{K}}  |\textbf{v}_s (x,s)|^2 \, \text{d}x\text{d}s.  
	\end{aligned} 
	\label{eq:ineq-3} 
\end{equation}
Finally, putting together \eqref{eq:ineq-2} and \eqref{eq:ineq-3}, since $\textbf{v} = \textbf{u} \, e^{-\gamma t}$, we obtain \eqref{eq:energy-est}.

\end{proof}

\section*{Acknowledgments}
Lorenzo Baldassari was supported by the DOE under grant DE-SC0020345. Maarten V. de Hoop was supported by the Simons Foundation under the
MATH + X program, the National Science Foundation under grant DMS-2108175,
and the corporate members of the Geo-Mathematical Imaging Group at Rice University. Elisa Francini and Sergio Vessella were partially supported by the Italian Ministry of Education, University and Research (MIUR) research project 201758MTR2, Prin 2017.

\bibliographystyle{siam}
\bibliography{references}

\end{document}